\documentclass[11pt,letterpaper,reqno]{amsart}
\usepackage[left=30mm,top=30mm,right=30mm,bottom=30mm]{geometry}
\usepackage{amssymb, amsthm, array,subcaption,multicol,times,color,enumerate,graphicx,nicefrac,xcolor,mathrsfs}
\usepackage{tikz}
\usepackage{comment}
\usepackage{hyperref}
\usepackage{url}
\usepackage[ruled,vlined]{algorithm2e}
\hypersetup{
 colorlinks=true, 
 linkcolor=blue, 
 citecolor=blue, 
 filecolor=blue, 
 urlcolor=blue 
}
\usepackage[nameinlink,capitalize,noabbrev]{cleveref}
\Crefname{notation}{Notation}{Notations}
\Crefname{AlgoLine}{Line}{Lines}
\usepackage{tikz}
\usetikzlibrary{arrows.meta,positioning,shapes,calc,decorations.pathreplacing}

\newtheorem{theorem}{Theorem}
\newtheorem*{theorem*}{Theorem}

\newtheorem{proposition}[theorem]{Proposition}

\newtheorem{conjecture}[theorem]{Conjecture}

\newtheorem*{claim*}{Claim}

{\theoremstyle{remark}
 \newtheorem{remark}[theorem]{Remark}}
 \newtheorem*{remark*}{Remark}
{\theoremstyle{definition}
 \newtheorem{definition}[theorem]{Definition}
 \newtheorem{notation}[theorem]{Notation}
 \newtheorem{example}[theorem]{Example}
}

\newcommand{\K}{\mathbf{k}}

\begin{document}

\title{On a computational approach to the Nash blowup problem}


\author[{F. Castillo}]{Federico Castillo}
\address{Centro de Investigaci\'on en Matem\'aticas, Guanajuato, M\'exico.}  
\email{federico.castillo@cimat.mx}

\author[D. Duarte]{Daniel Duarte}
\address{Centro de Ciencias Matematicas, UNAM, Morelia, M\'exico}
\email{adduarte@matmor.unam.mx}

\author[M. Leyton-Alvarez]{Maximiliano Leyton-\'Alvarez} %
\address{Instituto de Matem\'aticas, Universidad de Talca, Talca, Chile} %
\email{mleyton@utalca.cl}

\author[A. Liendo]{Alvaro Liendo} %
\address{Instituto de Matem\'aticas, Universidad de Talca, Talca, Chile} %
\email{aliendo@utalca.cl}

\date{\today}

\thanks{
\mbox{\hspace{11pt}}{\it Key words}: Resolution of singularities, Nash blowup, Toric varieties.\\
\mbox{\hspace{11pt}} The first author was partially supported by Fondecyt project 1260970. The second author was partially supported by CONAHCYT project CF-2023-G-33. The third author was partially supported by Fondecyt project 1221535. Finally, the fourth author was partially supported by Fondecyt project 1240101.}

\begin{abstract}
In this paper we describe the implementation that led to the counterexamples to the Nash blowup conjectures recently discovered by the authors. We also provide new examples of toric varieties with prescribed singularities that are not resolved by the normalized Nash blowup, including cyclic quotient singularities, toric hypersurfaces, and Q-factorial Gorenstein singularities. In addition, we report extensive computational evidence: tens of thousands of two-dimensional toric varieties that are resolved by iterating the Nash blowup, and millions of three-dimensional toric varieties that are resolved by iterating the normalized Nash blowup. This provides positive evidence for the remaining open cases of the conjectures.
\end{abstract}

\maketitle

\section*{Introduction}

The problem of resolution of singularities is a central topic in algebraic geometry \cite{kollar2009lectures}. While resolutions are known to exist for varieties over characteristic zero fields \cite{hironaka1964resolution}, the  positive characteristic case has resisted many attempts over time. 

In the early sixties, John Nash proposed a method, known as Nash blowups, for approaching the resolution problem. Informally, the Nash blowup replaces the singular points of a variety by the tangency at nearby non-singular points \cite{Nob75}. It was proposed to resolve singularities by iterating this process alone or by composing it with the normalization \cite{Se54,GS1,Sp90}.

Over the past fifty years, many results have supported Nash’s idea (see, for instance, \cite{GS2,Hi83,Sp90,DDR}). However, in our recent works \cite{CDLAL,CDLL3d} we proved that, in general, the answer to Nash’s question is negative: there exist singular algebraic varieties for which the iteration of Nash blowups or normalized Nash blowups never produce a resolution. Notably, the counterexamples were found within the realm of toric varieties.

Ever since their introduction in the 1970s, toric varieties have offered a rich environment for investigating broader theories, largely due to the two-way correspondence they establish between polyhedral combinatorics and algebraic geometry. When combined with modern computational power, toric varieties have become an abundant source of explicit examples in algebraic geometry. In particular, the counterexamples to Nash’s question mentioned above were discovered through extensive computer experimentation.

In order to turn Nash's original question into a computational problem, we built on prior works \cite{GS1,GrMi12,GoTe14,DJNB} that provide a combinatorial description of the Nash blowup for toric varieties. Thanks to this description, one can translate the Nash blowup process into an algorithm on the combinatorial data defining a toric variety. We implemented this algorithm in the computer algebra system SageMath \cite{sagemath} and used it to systematically explore the conjectures. Through this computational approach, we identified the first examples where the Nash blowup process fails to resolve a singularity, instead entering a loop.

In this paper we present the algorithms and implementation that led to these findings.  In addition to the original counterexamples we previously announced in \cite{CDLAL,CDLL3d}, we exhibit new examples of toric singularities with special properties, including cyclic quotient singularities, toric hypersurface singularities, and $\mathbf{Q}$-Gorenstein singularities, that are not resolved by Nash blowups. We also report the positive evidence obtained from our experiments: in millions of tested toric varieties, we found that the iteration of Nash blowups does resolve the singularities of 2-dimensional toric varieties and the iteration of normalized Nash blowups does resolve the singularities of 3-dimensional normal toric varieties.

Some important remarks about our approach to searching for examples are in order. We adopt the usual notation for toric varieties: $\mathsf{N}$ and $\mathsf{M}$ are lattices dual to each other, and normal toric varieties are defined by cones in $\mathsf{N_{\mathbf{R}}}$. Recall that the classical method for resolving toric varieties consists of subdividing a cone $\sigma\subset\mathsf{N_{\mathbf{R}}}$ using stellar subdivisions. On the other hand, the combinatorial description of the Nash blowup can be formulated both in $\mathsf{N}$ and in $\mathsf{M}$. A search for examples on the $\mathsf{N}$-side was previously carried out in \cite{Ataetal}, where no counterexamples were found. In contrast, our search was entirely conducted on the $\mathsf{M}$-side. We adopted this viewpoint because it provides a unified framework for both the normalized and non-normalized Nash blowups. This seemingly slight shift proved decisive: it was in $\mathsf{M}$ that we found the counterexamples.

Another fundamental aspect of our strategy is that we mainly concentrated the initial search within the family of Reeves cones, viewed as cones in $\mathsf{M_{\mathbf{R}}}$ (see Definition \ref{def:reeves}). We chose this family for exploring Nash’s question because it is known in polyhedral geometry for exhibiting pathological behavior \cite{reeve1957volume,beck2015very}. It was through iterating the Nash blowup, or its normalized version, on the toric variety defined by a Reeves cone in dimension four, that we found the first counterexamples to Nash's questions.

The paper is organized as follows. In Section 1 we review background on toric varieties and fix notation. Section 2 introduces the algorithm for computing the Nash blowups in the toric setting. Section 3 describes our implementation of this algorithm and presents the negative results: the explicit counterexamples and how they were found through computation. In Section 4, we discuss the positive results, namely the large collection of cases where iterated Nash blowups do achieve a resolution. 

This paper was announced in \cite{CDLAL} and \cite{CDLL3d} as the natural complement to the results presented there. Our study of Nash blowups of toric varieties, and the results we obtained, echoes Fulton’s famous remark: \emph{toric varieties have provided a remarkably fertile testing ground for general theories} \cite{fulton1993introduction}.

\subsection*{Acknowledgments}

This collaboration began at the \href{https://sites.google.com/view/agrega0/home}{AGREGA} workshop, which took place at Universidad de Talca in January 2024 where the second named author presented the Nash blowup conjectures. We extend our gratitude to the institution for its support and hospitality.

\section{Preliminaries}

We assume familiarity with the toric geometry notation in \cite{cox2011toric} and with the polyhedral methods of \cite{schrijver1998theory}.
For completeness, we briefly recall the relevant background and notation.

\subsection{Cones}
\label{Cones}
Let $ \mathsf{M} $ and $ \mathsf{N} $ be lattices of rank $ n $ dual to each other.
We denote the duality pairing as $\langle \cdot , \cdot\rangle\colon \mathsf{M} \times \mathsf{N} \to \mathbf{R} $. A rational polyhedral cone $ \sigma \subseteq \mathsf{N}_{\mathbf{R}} := \mathsf{N} \otimes_{\mathbf{Z}} \mathbf{R} $ is a set of the form
\begin{equation*}\label{eq:internal}
\sigma = \left\{ \sum_{u\in G} \lambda_u u ~|~ \lambda_{u} \in \mathbf{R}_{\geq 0} \right\},
\end{equation*}
where $ G \subseteq \mathsf{N}$ is a finite set that we refer to as the \textit{generators} of $ \sigma $.
If no proper subset of $ G $ generates the same cone, and its elements are primitive, we say the set $ G $ is a set of minimal generators.
A cone $ \sigma $ is pointed if $ \sigma \cap -\sigma = \left\{ \mathbf{0} \right\} $.

By the Weyl-Minkowski Theorem, every rational polyhedral cone $ \sigma $ can be alternatively written as
\begin{equation*}\label{eq:external}
	\sigma = \left\{ u \in \mathsf{N}_{\mathbf{R}} ~|~\langle m , u \rangle \geq 0, \forall m \in F \right\},
\end{equation*}
where $ F \subseteq \mathsf{M} $ is a finite set.
If no proper subset of $ F $ defines the same cone, we say $ F $ is minimal.

\begin{notation} \label{not:1}
Throughout this paper, we pick explicit isomorphisms from $ \mathsf{M} $ and $ \mathsf{N} $ to $ \mathbf{Z}^{n} $ such that the duality pairing corresponds to the usual dot product. Letting $ \mathbf{A} $ be an $ n \times m $ integral matrix, we define $ \operatorname{Cone}(\mathbf{A})\subset \mathsf{N}_\mathbf{R}$ to be the cone generated by the columns of $\mathbf{A}$. In this case the Weyl-Minkowski Theorem states that there exists an integral matrix $ \mathbf{B} $ such that $ \operatorname{Cone}(\mathbf{A}) = \left\{ u\in \mathsf{N}_\mathbf{R}\mid  \mathbf{B}u\geq 0 \right\}$.
\end{notation}

\begin{example}\label{ex:running}
It can be computed using SageMath that the following descriptions are minimal, i.e., on the left we cannot delete any column and on the right we cannot delete any row.
\begin{tiny}
\begin{equation*}
\sigma = \operatorname{Cone}
\left(\left[\begin{array}{rrrrrr}
-2 & -1 & 5 & 0 & 0 & 5 \\
5 & 3 & 4 & -1 & 1 & 1 \\
1 & 2 & -1 & 1 & 4 & -2 \\
2 & -1 & 1 & 2 & 0 & -2
\end{array}\right]\right)
=
\left\{
	u \in \mathbf{R}^{4} \,\middle|\,
	\left[\begin{array}{rrrr}
15 & 8 & -2 & 5 \\
15 & 5 & 1 & 2 \\
2 & 4 & -1 & 8 \\
41 & -11 & 57 & 40 \\
-2 & 40 & -10 & 25 \\
9 & 5 & 45 & -20 \\
3 & 1 & 13 & -6 \\
5 & 1 & 21 & -8
\end{array}\right]
	\left[\begin{array}{r}
 u_1 \\
 u_2 \\
 u_3 \\
 u_4 \end{array}\right]
\geq
	\left[\begin{array}{r}
0 \\
0 \\
0 \\
0 \\
0 \\
0 \\
0 \\
0 
\end{array}\right]
\right\}.
\end{equation*}
\end{tiny}
Notice that in this example there are six generators and eight inequalities.
\end{example}

Let $ \sigma \subseteq \mathsf{N}_{\mathbf{R}} $ be a cone and $ m \in \mathsf{M}$ be such that $\langle m , u \rangle \geq 0$ for every $ u \in \sigma $.
The subset $ \left\{ u \in \sigma \mid\langle m , u \rangle = 0 \right\}  \subseteq \sigma$ is a face of $\sigma$. A facet is a face of codimension one. Associated to $ \sigma $ we define another cone, its dual, as
\begin{equation*}
	\sigma^{\vee} = \left\{ m \in \mathsf{M}_{\mathbf{R}} ~|~\langle m , u \rangle \geq 0, \forall u\in \sigma \right\}.
\end{equation*}
The dual cone is again rational polyhedral.
From now on by a cone $\sigma\subset \mathsf{N}_\mathbf{R}$ we mean a rational polyhedral cone which is also pointed.

\begin{example}
    With \cref{not:1}, the dual of a cone $\sigma=\operatorname{Cone}(\mathbf{A})$ is $\sigma^\vee=\operatorname{Cone}(\mathbf{B}^{\mathsf{T}})$. Hence, the dual $\sigma^\vee$ of the cone $\sigma$ in \cref{ex:running} is generated by the rows of the matrix defining the inequalities. 
\end{example}

\subsection{Unimodular equivalence}

The group $ \operatorname{Aut}(\mathsf{N}) \subseteq \operatorname{GL}(\mathsf{N}_{\mathbf{R}})$ is isomorphic to $\operatorname{GL}_n(\mathbf{Z})$ and its elements are called unimodular transformations. Most of the operations we are interested in are invariant under unimodular transformations. We define two cones $ \sigma $ and $ \sigma ' $ to be \textit{unimodularly equivalent}, if there exists a unimodular transformation $ \mathbf{U} $ such that $ \mathbf{U}\sigma = \sigma ' $. We denote this by $\sigma\simeq \sigma'$.
For instance, we have
\begin{equation}\label{eq:equivalence}
\operatorname{Cone}\left(
\left[\begin{array}{rrrrrr}
-2 & -1 & 5 & 0 & 0 & 5 \\
5 & 3 & 4 & -1 & 1 & 1 \\
1 & 2 & -1 & 1 & 4 & -2 \\
2 & -1 & 1 & 2 & 0 & -2
\end{array}\right]\right)
\simeq
\operatorname{Cone}\left(
\left[\begin{array}{rrrrrr}
1 & 0 & 5 & 3 & 4 & 1 \\
0 & 1 & 9 & 4 & 8 & 4 \\
0 & 0 & 24 & 2 & 5 & 17 \\
0 & 0 & 0 & 8 & 11 & -6
\end{array}\right]\right),
\end{equation}
given by the unimodular transformation represented by the matrix
\begin{equation*}
\mathbf{U}
=
\left[\begin{array}{rrrr}
1 & 0 & 1 & 1 \\
2 & 0 & 2 & 1 \\
4 & 1 & 1 & 1 \\
1 & -1 & 3 & 2
\end{array}\right].
\end{equation*}
The matrix on the right hand side of  \eqref{eq:equivalence} is the Hermite normal form (defined immediately below) of the matrix on the left hand side.

\begin{definition}
Let $ \mathbf{A} $ be a $ n' \times n $ matrix with integer entries.
We say that $ \mathbf{A} $ is in Hermite normal form if
\begin{enumerate}
	\item On each non-zero row the \emph{pivot}, first non-zero entry from left to right, is located to the right of the pivot of the row above.
	\item Each pivot is positive.
	\item Elements above each pivot are nonnegative and strictly smaller than the pivot.
\end{enumerate}
\end{definition}

This is the integral version of the notion of \emph{row-echelon} form in linear algebra over a field.
We can reduce an integer matrix $ \mathbf{A} $ to a Hermite normal form by doing elementary operations on the rows that do not change the absolute value of the determinant.
In matrix terms, we can find an $ n \times n $ matrix $ \mathbf{U} $ with determinant $ \pm 1 $ such that $ \mathbf{U} \mathbf{A} $ is in Hermite normal form.
Moreover, any two unimodularly equivalent matrices have the same Hermite normal form.

However, interchanging the \emph{columns} of $ \mathbf{A} $ changes the integral span of its rows, so it changes the Hermite normal form.
The cone generated by the columns is clearly independent of the order of its generators, so we cannot naively talk about the normal Hermite form of a cone of the form $ \operatorname{Cone}(\mathbf{A}) $.
To obtain a canonical form, one must choose a particular order of the generators (columns of the matrix).

\begin{remark}\label{Hermite factorial}
Theoretically, one could consider all $n!$ permutations of the columns, compute each corresponding Hermite normal form, and choose the lexicographically largest matrix. While this defines a canonical form, it also highlights a major computational drawback: a canonical form is only meaningful if it can be computed in a reasonable amount of time.
\end{remark}

\begin{definition}
A cone $ \sigma \subseteq\mathsf{N}_{\mathbf{R}}  $ is simplicial if it is generated by a basis of the vector space $ \mathsf{N}_\mathbf{R} $. Furthermore, we say that $\sigma$ is unimodular if it is generated by a basis of the lattice $ \mathsf{N} $.
\end{definition}

If $\operatorname{Cone}(\mathbf{A})$ is unimodular and $\mathbf{A}$ contains a minimal set of generators, then the Hermite normal form of $\mathbf{A}$ is the identity.

\subsection{Affine semigroups}

An affine semigroup $ S \subseteq \mathsf{M}$ is a semigroup generated by finitely many elements. We say that $ S $ is \emph{pointed} if for every nonzero element $ m \in S $ we have $ -m \notin S $. 

\begin{definition}
    We say that a pointed affine semigroup $ S $ is unimodular if it is unimodularly equivalent to the affine semigroup $ \mathbf{Z}_{\geq 0} \left\{ \mathbf{e}_{1}, \dots, \mathbf{e}_{n} \right\} \subseteq \mathsf{M}$.
\end{definition}

Equivalently, an affine semigroup $S$ is unimodular if  $S=\sigma^\vee\cap \mathsf{M}$ for some unimodular cone $\sigma^\vee$.

Every pointed affine semigroup has a unique minimal generating set $ \mathcal{H} $. This set can be characterized as the set of indecomposable elements, that is, $ h \in \mathcal{H} $ if and only if $ h $ cannot be written as the sum of two non-zero elements of $ S $. Without loss of generality, up to changing the lattice $\mathsf{M}$, we may and will always assume that $\mathbf{Z}S=\mathsf{M}$.

\begin{definition}
The minimal generating set of a pointed affine semigroup $S$ is called the Hilbert basis.
\end{definition}

An affine semigroup is \emph{saturated} if $ k\cdot m \in S$ , with $m\in\mathsf{M}$ and $k$ a positive integer, implies that $ m \in S $.
Saturated affine semigroups are in correspondence with cones $\sigma^\vee\subset \mathsf{M}_{\mathbf{R}} $ by letting $ S $ be the semigroup $ \sigma^\vee\cap \mathsf{M} $. This correspondence maps pointed semigroups to pointed cones. When dealing with saturated pointed semigroups, we will often confuse the cone with the semigroup, and thus we can talk about the Hilbert basis of a pointed cone $\sigma^\vee\subset \mathsf{M}_\mathbf{R}$. 

\begin{remark}
Computing Hilbert bases is the most involved part in our computations. Theoretically, little is known about them, and computing the Hilbert basis of a pointed semigroup is a hard computational problem. However, due to their importance in discrete optimization, specialized software for computing them is available. In particular, we use the software Normaliz \cite{Normaliz} to compute Hilbert bases of pointed cones.  
\end{remark}

\begin{example}
In the cone from Example \ref{ex:running} the Hilbert basis of the corresponding semigroup has 74 elements.
\end{example}

\section{Nash blowups of toric varieties}

We now briefly describe the Nash blowup of an affine variety over an algebraically closed field $\K$ of any characteristic. However, since our main interest lies in the computational aspects of Nash blowups of toric varieties, the details from algebraic geometry may be skipped: just take \cref{def:nash_modified_nonnormal} to be the definition of the Nash blowup of a toric variety and \cref{def:nash_modified} to be the definition of the normalized Nash blowup of a normal toric variety.

Let $X \subseteq \K^n$ be an equidimensional algebraic variety of dimension $d$, where $\K$ is an algebraically closed field. Consider the Gauss map:
\begin{align*}\label{eq:gauss}
\Phi\colon X \setminus \operatorname{Sing}(X) \to \operatorname{Grass}(d, n) \quad \text{defined by} \quad x \mapsto T_x X\,,
\end{align*}
where $\operatorname{Grass}(d, n)$ denotes the Grassmannian of $d$-dimensional vector spaces in $\K^n$, and $T_x X$ is the tangent space to $X$ at $x$. 
Let $X^*$ be the Zariski closure of the graph of $\Phi$ and $\nu\colon X^* \to X$ be the composition of the inclusion $X^*\hookrightarrow X\times\operatorname{Grass}(d, n)$ and the projection onto the first coordinate. The morphism $\nu$ is a proper birational map that is an isomorphism over $X \setminus \operatorname{Sing}(X)$. The pair $(X^*,\nu)$ is called the \emph{Nash blowup} of $X$. The composition $\eta\circ\nu$, where $\eta$ is the normalization map, is called the \emph{normalized Nash blowup} of $X$.

The \emph{Nash blowup conjecture} asks whether successive Nash blowups of a singular variety eventually produce a non-singular variety. It has been known since the 1970s that this conjecture has a negative answer in positive characteristic, and thus we restrict our attention to characteristic zero. Likewise, the \emph{normalized Nash blowup conjecture} asks whether iterating the normalized Nash blowup of a singular variety eventually yields a non-singular variety. Prior to our work, this conjecture remained open in all characteristics and in dimension three and higher.

\medskip

An \emph{affine toric variety} is an algebraic variety of the form $\operatorname{Spec}\K[S]$, where $S$ is a pointed affine semigroup and $\K[S]$ is the semigroup algebra
\[
\K[S]=\bigoplus_{u\in S}\K\cdot\chi^{u},\quad\mbox{with}\quad\chi^{0}=1,\mbox{ and }\chi^{u}\cdot\chi^{u'}=\chi^{u+u'},\ \forall u,u'\in S\,.
\]
With this definition, an affine toric variety is non-singular if and only if the corresponding semigroup is unimodular. Moreover, an affine toric variety is normal if and only if the corresponding semigroup in $\mathsf{M}$ is saturated and so corresponds to a pointed cone in $\mathsf{M}_\mathbf{R}$.

The Nash blowup of an affine toric variety can be described in purely combinatorial terms. This approach was originally introduced by G. González-Sprinberg for complex normal toric varieties \cite{GS1} and was later refined by removing the hypothesis of normality and for varieties over fields of arbitrary characteristic, by Grigoriev-Milman and González-Teissier, and Duarte-Jeffries-Núñez Betancourt, respectively \cite{GrMi12,GoTe14,DJNB}. 

The resulting combinatorial description is presented in the following subsections. We describe both the combinatorial algorithm for the Nash blowup and for the normalized Nash blowup. Moreover, we also discuss some subtleties regarding its computational implementation.

\subsection{Algorithm for the Nash blowup}

\begin{definition}[Nash blowup]
\label{def:nash_modified_nonnormal}
Let $ S \subseteq \mathsf{M}	$ be an affine semigroup and $ \mathbf{k} $ an algebraically closed field. The Nash blowup of $ S $ is the collection of affine semigroups obtained in the following process.

\begin{description} 
	\item[D1] Compute the Hilbert basis $ \mathcal{H} $ of $S$.
	\item[D2] Compute the set $ \mathscr{B} = \left\{ I \subseteq \mathcal{H} \mid I \text{ is a basis for }\mathsf{M}\otimes_\mathbf{Z}\mathbf{k}  \right\} $ of all linear bases contained in $\mathcal{H}$.
	\item[D3] For each $I \in \mathscr{B}$, we compute
    \[
    \mathcal{G}_I = \mathcal{H}\cup \bigcup_{h\in I} \mathcal{G}_I(h),\quad \text{where}\quad \mathcal{G}_I(h) := \big\{ 
    g-h \mid g\in \mathcal{H}\setminus I \text{ and } \left( I\cup \{g\} \right ) \setminus \{h\} \in \mathscr{B}
    \big\}.
    \]
    \item[D4] Return the collection of semigroups $S_I$ generated by the sets $\mathcal{G}_I$ that are pointed.
\end{description}
Given a collection of semigroups, the Nash blowup is the union of the Nash blowups of each of the semigroups.
\end{definition}

In algebraic geometry terms, the set of affine toric varieties $\operatorname{Spec}\mathbf{k}[S_I]$, with $S_I$ pointed, provides an affine cover of the Nash blowup of the toric variety $\operatorname{Spec}\mathbf{k}[S]$. 

\begin{remark}
In the first step, it is enough to take any set of generators of the semigroup instead of the Hilbert basis, but to make the second step more efficient it is worthwhile to take the Hilbert basis of the semigroup. The second step is a simple determinant computation modulo the characteristic of the base field. Since we are computing in low dimensions, the determinant computation is fast. It is worth mentioning that it is only at this step that the base field intervenes: as an output, we obtain a different set $\mathscr{B}$ for different characteristics.
	
The third step involves only basic operations, but the number of operations grows exponentially with the cardinality of $\mathcal{H}$. Finally, the fourth step can be easily achieved by checking if the cone generated by $\mathcal{G}_I$ is pointed. In this step, it is wise to  `clean' the generators $ \mathcal{G}_I $ of $S_I$ to obtain a Hilbert basis. This facilitates further iterations of Nash blowups as we already have a Hilbert basis as an input for the first step in the next iteration.
\end{remark}

We now work out a couple of examples in detail.

\begin{example}[{\cite[Example 1]{Nob75}}]
	\label{ex:nobile}
Let $ S \subseteq \mathbf{Z} $ be the semigroup generated by $ \left\{ 2,3 \right\} $ and $\mathbf{k}$ an algebraically closed field of characteristic three. The semigroup algebra $ \mathbf{k}[S] $ is isomorphic to $ \mathbf{k}[x,y]/(x^3-y^2) $, the affine coordinate ring of a plane curve with a cusp singularity in the origin. Let us compute the Nash blowup of $S$. Only the singleton $I=\left\{ 2 \right\} $ is a basis for $\mathbf{Z}\otimes_\mathbf{Z}\mathbf{k}\simeq \mathbf{k} $. Its corresponding $\mathcal{G}_I$ is equal to $ \left\{ 2, 3 \right\} \cup \emptyset$ since we cannot replace $ 2	$ by $ 3 $ to obtain a basis. Hence the Nash blowup of $S$ is $\{S\}$. This example shows that iterating Nash blowup does not resolve the singularities in positive characteristic. 
\end{example}

Due to \cref{ex:nobile}, in the sequel, we only consider base fields $\K$ of characteristic zero when considering the Nash blowup of a semigroup.

\begin{example}[Whitney umbrella]
	\label{ex:unmbrella}
	Let $ S \subseteq \mathsf{M}$, with $\mathsf{M}=\mathbf{Z}^{2}$ be the semigroup generated by the columns of 
	\begin{equation*}
		\begin{bmatrix}
			1&1&0 \\
			1&0&2 
		\end{bmatrix}.
	\end{equation*}
	The semigroup algebra $ \mathbf{k}[S] $	is isomorphic to $ \mathbf{k}[x,y,z]/(x^2-y^2z) $, the affine coordinate ring of the algebraic variety usually known as the Whitney umbrella. Let us compute the Nash blowup of $S$. Every pair is a basis and we obtain the following
	\begin{equation*}
		\mathcal{G}_{12} =
		\begin{bmatrix}
			1& 1& 0&-1&-1 \\
			1& 0& 2& 2& 1 
		\end{bmatrix},
		\mathcal{G}_{13} =
		\begin{bmatrix}
			1& 1& 0& 0&-1 \\
			1& 0& 2&-1&-2
		\end{bmatrix},
		\mathcal{G}_{23} =
		\begin{bmatrix}
			1& 1& 0& 0& 1 \\
			1& 0& 2& 1&-1 
		\end{bmatrix},
	\end{equation*}
	The one in the middle is not pointed as it contains both $[0,2]^T$ and $ [0,-2]^T$.
	For the other two we can compute the Hilbert Basis by repeatedly discarding redundant elements, as semigroups we obtain
	\begin{equation*}
		\mathbf{Z}_{\geq0}\mathcal{G}_{12} = \mathbf{Z}_{\geq0}
		\begin{bmatrix}
			1&-1 \\
			0& 1 
		\end{bmatrix},\qquad
		\mathbf{Z}_{\geq0}\mathcal{G}_{23} = \mathbf{Z}_{\geq0}
		\begin{bmatrix}
			0& 1 \\
			1&-1 
		\end{bmatrix}.
	\end{equation*}
	both of which are unimodular semigroups.
	Since these two charts provide an affine cover of the Nash blowup of the Whitney umbrella, we conclude that the resulting Nash blowup is smooth.
\end{example}

To organize this process we consider an infinite directed graph (a digraph). We say that an affine semigroup is a \emph{child} of $S$ if it occurs in the result of a single Nash blowup of $S$.

\begin{definition}[Nash digraph]
	\label{def:nash_digraph}
The Nash digraph is the infinite directed graph $ \mathcal{S}(\mathsf{M}) $ whose vertex set is the set of all affine semigroups $S$ of full rank in $ \mathsf{M} $ modulo unimodular equivalence,
and there is a directed edge from $S$ to each of its children.
\end{definition}

Up to unimodular equivalence, there is a unique unimodular semigroup that we denote by $\epsilon$.

\begin{conjecture}[Nash blowup conjecture]\label{conj:tree}
Every directed path in $ \mathcal{S}(\mathsf{M})$ eventually reaches $ \epsilon$.
\end{conjecture}

\subsection{Normalized Nash blowup}
We now describe the normalized Nash blowup of the normal toric variety corresponding to the pointed saturated semigroup $S$ so that $S=\sigma^\vee\cap \mathsf{M}$ for some pointed cone $\sigma^\vee\subset \mathsf{M}_\mathbf{R}$. We need one more definition.
\begin{definition}
	\label{def:feasible_cone}
	Let $ \mathsf{P}$ be a polyhedron and $\mathbf{v}$ a vertex of $\mathsf{P}$.
	We define the \emph{feasible cone at $\mathbf{v}$} as 
	\begin{align*}
		\operatorname{fcone}(\mathbf{v}, \mathsf{P}) &= \operatorname{Cone}\left(\left\{ e \mid e \text{ is the direction of an edge adjacent to }\mathbf{v}\text{ pointing away from it} \right\}\right) \\
        &=\operatorname{Cone}\left(\mathsf{P}-\mathbf{v}\right)\,.
	\end{align*}
\end{definition}
The translation $ \mathbf{v} + \operatorname{fcone}(\mathbf{v}, \mathrm{P}) $ is called the tangent cone.
Feasible cones are dual to the more familiar normal cones. Now we are ready to define the normalized Nash blowup.

\begin{definition}[Normalized Nash blowup]
\label{def:nash_modified}
Let $ \sigma^\vee\subset \mathsf{M}_\mathbf{R}$ be a pointed cone and $ \mathbf{k} $ an algebraically closed field. The normalized Nash blowup of $\sigma^\vee$ is the collection of pointed cones obtained in the following process.

\begin{description} 
\item[N1] Compute the Hilbert basis $ \mathcal{H} $ of $S = \sigma^\vee\cap \mathsf{M}$.
\item[N2] Compute the set $\displaystyle \mathcal{B} = \left\{ \sum_{h \in B} h \mid B \subseteq \mathcal{H} \text{ is a basis for }\mathsf{M}\otimes_\mathbf{Z}\mathbf{k}  \right\} $.
\item[N3] Return the set of feasible cones of the vertices of the unbounded polyhedron $\operatorname{Conv}(\mathcal{B}) + \sigma^\vee$.
\end{description}
Given a collection of cones, the normalized Nash blowup is the union of the normalized Nash blowups of each of the cones.
\end{definition}

\begin{remark}
Let us highlight some computational aspects. The bottleneck continues to be computing the Hilbert basis in the first step. The second step is a simple determinant computation modulo the characteristic of the base field and again it is only at this step that the characteristic of the base field intervenes. For the third step, computing feasible cones, convex hulls and Minkowski sums of polyhedra in lower dimensions is quite efficient.
\end{remark}

Similar to \cref{def:nash_digraph}, we organize this process in a digraph. In this case we need to consider the dependence on the characteristic of the base field $\K$. We say that a cone is a child of $\sigma^\vee$ if it occurs in the result of a single Nash blowup of $\sigma^\vee$, this notion depends on the characteristic of the base field $\K$.

\begin{definition}[Normalized Nash digraph]
	\label{def:normalized_nash_digraph}
The normalized Nash digraph of characteristic $p$ 
is the infinite directed graph $ \mathcal{K}_{p}(\mathsf{M})$ whose vertex set is the set of all pointed full dimensional polyhedral cones in $\mathsf{M}_{\mathbf{R}} $ modulo unimodular equivalence, and there is an edge from a cone to each of its children.
\end{definition}

Up to unimodular equivalence, there is a unique unimodular cone that we denote by $\epsilon$.

\begin{conjecture}[Normalized Nash blowup conjecture]
\label{conj:normalized_tree}
Every directed path $ \mathcal{K}_{p}(\mathsf{M}) $ eventually reaches $ \epsilon$.
\end{conjecture}

As stated in \cite{Ataetal}, \cref{conj:tree} and \cref {conj:normalized_tree} could fail in two possible ways: either with the existence of an infinite path that never reaches $\epsilon$, or with the existence of a directed cycle. We disprove both of these conjectures by finding directed cycles in both digraphs.

\section{Implementation and computational results}

In this section, we present the computational implementation that led us to the counterexamples to the Nash conjecture and the normalized Nash conjecture, first given in \cite{CDLAL, CDLL3d}. The SageMath code necessary to reproduce the results is available at \cite{github}. 

The exposition in \cite{CDLAL, CDLL3d} focuses exclusively on providing proofs of the counterexamples without the use of computer assistance. In contrast, in this paper we present results that can only be established through computational methods. Specifically, we provide new counterexamples to the normalized Nash blowup conjecture consisting of simplicial toric varieties, toric hypersurfaces, and $\mathbf{Q}$-Gorenstein toric singularities. The only remaining open case of the Nash blowup conjecture is in dimension two, while for the normalized Nash blowup conjecture it is in dimension three. We provide positive computational evidence supporting these conjectures in the toric setting.

\medskip

We will represent the vertices of the digraphs $\mathcal{S}(\mathsf{M})$ and $ \mathcal{K}_p(\mathsf{M}) $ with elements of $ \operatorname{Mat}_{n\ast}( \mathbf{Z})$, the set of integer matrices with $n$ rows and a variable number of columns greater or equal to $n$. For a cone $\sigma^\vee\in \mathcal{K}_p(\mathsf{M})$, we take the ray generators of the cone as columns  and for a semigroup $S\in \mathcal{S}(\mathsf{M})$, we take its Hilbert basis as columns as well.

To deal with unimodular equivalence we use the Hermite normal form of the matrix containing the  primitive ray generators of a cone. However, we are still faced with the problem of efficiently computing a unique Hermite normal form up to reordering of the rays of a cone (see Remark \ref{Hermite factorial}). To address this, we apply an algorithm developed by Kreuzer and Skarke, implemented in their software PALP \cite{kreuzer2004palp}, which efficiently produces a canonical form of polytopes (see also \cite[Appendix A]{grinis2013normal} for a detailed explanation of the algorithm). Their method exploits the vertex–facet incidence graph to determine a canonical ordering of the vertices in a systematic way. This approach can be readily adapted to the case of pointed polyhedral cones.
The algorithm was previously implemented for polytopes in SageMath, and we have extended this implementation to handle cones.

\begin{definition}[PALP normal form]\label{def:representative}\
\begin{enumerate}
\item Let $\sigma^\vee \subset \mathsf{M}_\mathbf{R}$ be a pointed full-dimensional cone. Its \emph{PALP normal form} is the matrix $\mathbf{A}$ obtained as the Hermite normal form of the matrix whose columns are the primitive generators of the rays of $\sigma^\vee$, arranged in the unique order prescribed by PALP.
    
\item  Let $S\subset\mathsf{M}$ be a (not necessarily saturated) affine semigroup. The convex hull of $S$ is a cone $\sigma^\vee\subset \mathsf{M}_\mathbf{R}$ and so it has a unique PALP normal form and there is a unimodular transformation $\mathbf{U}$ sending $\sigma^\vee$ to its PALP normal form. Let $\mathcal{H}$ be the Hilbert basis of $S$. A \emph{PALP normal form} of $S$ is the matrix $\mathbf{A}$ whose columns are the elements of $\{ \mathbf{U}h\mid h\in\mathcal{H}\}$ ordered lexicographically.
\end{enumerate}
\end{definition}

\begin{example}
For the cone $\sigma^\vee$ of Example \ref{ex:running}, its PALP normal form is
\begin{equation*}
\mathbf{A}=\left[\begin{array}{rrrrrrr}
1 & 0 & 3 & 33 & -26 & 32 \\
0 & 1 & 4 & 22 & -18 & 18 \\
0 & 0 & 5 & 34 & -28 & 30 \\
0 & 0 & 0 & 49 & -37 & 51
\end{array}\right].
\end{equation*}
which is obtained from the Hermite normal form of the matrix obtained from $ \mathbf{A} $ by the permutation of the columns whose cycle description is $ (1,6,3)(2,4,5) $.

Let $S$ be the (not necessarily saturated) affine semigroup generated by the columns of $\mathbf{A}$. Let us find a PALP normal form of $S$. Firstly, the columns of $\mathbf{A}$ form the Hilbert basis of $S$. Since $\mathbf{A}$ is  the PALP normal form of the convex hull $\sigma^\vee$ of $S$, the desired unimodular transformation $\mathbf{U}$ in \cref{def:representative}~(2)  is the identity matrix. Hence, it only remains to order the columns of $\mathbf{A}$ using the lexicographic order. We conclude that the PALP normal form of $S$ is the matrix 
\begin{equation*}
 \left[\begin{array}{rrrrrrr}
-26  & 0 & 1 &3 &  32 &33 \\
-18 & 1 & 0 & 4 &  18 &22 \\
-28 & 0 & 0 & 5 & 30 & 34\\
-37 & 0 & 0 & 0 &  51 & 49
\end{array}\right].
\end{equation*}
\end{example}

\begin{remark}
The PALP normal form in \cref{def:representative} is unique only for cones. For a semigroup $S$, uniqueness no longer holds (\cref{ex:full-disclosure}). However, dealing with this lack of uniqueness does not significantly increase the computational complexity. This is due to the fact that full-dimensional cones seldom  have automorphisms, and even when they do the automorphism group is finite. Let us stress that this lack of uniqueness never mixes non-isomorphic semigroups. Indeed, the unimodular transformation $\mathbf{U}$ in \cref{def:representative}~(2) is unique up to composition with an automorphism of the cone $\sigma^\vee$ in PALP normal form. Hence, any two matrices arising as PALP normal forms of a given semigroup generate unimodularly equivalent semigroups, as illustrated in \cref{ex:full-disclosure}.
\end{remark}

\begin{example} \label{ex:full-disclosure}
Let $ S $ and $S'$ be the semigroups in $\mathbf{Z}^{2} $  generated by the columns of the matrices $\mathbf{A}$ and $\mathbf{A}'$, respectively, where
\begin{equation*}
	\mathbf{A}=\begin{bmatrix}
		0&1&1 \\
		2&0&1 
	\end{bmatrix},
    \quad\mbox{and}\quad
	\mathbf{A}'=\begin{bmatrix}
		0&1&2 \\
		1&1&0 
	\end{bmatrix}.
\end{equation*}
The cone $\sigma^\vee$ generated by both semigroups is the first quadrant that is already in PALP normal form. Hence in both cases the unimodular matrix in \cref{def:representative}~(2) is the identity. This shows that the chosen representatives of $S$ and $S'$ are their Hilbert bases, given by the columns of $\mathbf{A}$ and $\mathbf{A}'$, respectively, which are already in lexicographical order. Hence, we obtain that $S$ and $S'$ are both representatives of the same semigroup corresponding to the Whitney umbrella in \cref{ex:unmbrella} (actually, $S\cong S'$ via the only non-trivial isomorphism of $\sigma^\vee$ that corresponds to exchanging the first and second coordinates).
\end{example}

\begin{definition}    
We represent both the Nash digraph $\mathcal{S}(\mathsf{M})$ and the normalized Nash digraph $\mathcal{K}_p(\mathsf{M})$, by a function $\chi$ that assigns to each matrix $\mathbf{A} \in   \operatorname{Mat}_{n\ast}( \mathbf{Z} )$ corresponding to a cone or semigroup in its PALP normal form the \emph{set} of its children in PALP normal form. 
\end{definition}

\begin{example}
For example, for the normalized case we have
\begin{equation}
\chi \left(
	\left[
\begin{array}{rr}
1 & 3\\ 
0 & 5
\end{array}
	\right] \right)
=
\left\{ 
\left[
\begin{array}{rr}
	1 & 0\\
  0 & 1
\end{array}
	\right],
\left[
\begin{array}{rr}
	1 & 1\\
  0 & 3
\end{array}
	\right]
\right\}
\end{equation}
\end{example}

We now focus on how to compute the normalized Nash digraph $\mathcal{K}_{0}(\mathsf{M})$, i.e., in the case of characteristic zero. The cases of the digraphs $\mathcal{K}_{p}(\mathsf{M})$ and $\mathcal{S}(\mathsf{M})$  are analogous. We keep the partial knowledge of the digraph with the following two global sets:

\begin{enumerate}
  \item A set \texttt{VERTICES} of cones, each represented by its normal form in $ \operatorname{Mat}_{n\ast}( \mathbf{Z} ) $.
	\item A set \texttt{EDGES} of edges, given as a set of ordered pairs of cones.
\end{enumerate}
We can initialize the digraph with the single vertex containing the identity matrix of dimension $n$ representing $ \epsilon $, the unimodular cone and an edge to itself, i.e., a single loop. By repeated applications of normalized Nash blowup, we add a vertex with all of its descendants and the corresponding edges to the partial knowledge of the digraph.

The Nash digraph $ \mathcal{K}_{0}(\mathsf{M}) $ is infinite, but, assuming \cref{conj:normalized_tree}, the subgraph composed by all the descendants of a single element is finite. Under this assumption, we can compute this subgraph using breadth-first search, i.e., we explore the digraph one generation at a time: we first compute all the children of the given cone, then all the children of those cones, and so on. In practice, this amounts to storing the cones that are still to be processed in a queue, so that they are treated in the same order in which they were found.

\begin{algorithm}
\caption{Resolution subgraph of a cone}\label{algo:res_tree}
\KwIn{A cone $ \sigma^\vee $}
\KwOut{The whole subgraph of descendants of $ \sigma^\vee $ .}

Initialize queue $Q$ as empty\;
$ A \leftarrow \mathrm{NormalForm}(\sigma^\vee) $\;
\If {$ A \in \texttt{VERTICES} $}{
	exit\;
}
\Else {
	$Q.enqueue(A)$\;
}

\While{$Q$ is not empty}{
    $B \leftarrow Q.dequeue()$\;
		\If {$ B \in \texttt{VERTICES} $}{
			continue\;
		}
		$\texttt{VERTICES}.add(B)$\;
    \For{each vertex $C \in \chi( B )$}{
				$\texttt{EDGES}.add(B, C ) $\;
				$Q.enqueue(C)$\;
    }
}
\end{algorithm}

There is no guarantee that \cref{algo:res_tree} ends. If it does not end then the cone $\sigma^\vee$ is a counterexample to the normalized Nash conjecture in characteristic zero. If it does end, it updates the known part of $  \mathcal{K}_{0}(\mathsf{M})$ with all the descendants of $ \sigma^\vee$. Storing the portions of the digraph that have already been computed accelerates the overall computation by eliminating the need to recompute the corresponding normalized Nash blowups.

\subsection{Counterexample to Nash blowup conjecture in dimension 3 and higher}

We begin by reporting on our results for the non-normalized Nash blowup conjecture. 

We populated the digraph $\mathcal{S}(\mathsf{M})$ with $\mathsf{M}=\mathbf{Z}^3$ with 250 thousand vertices by running \cref{algo:res_tree} adapted to this case and we investigated the existence of cycles in it. To look for cycles inside a digraph, we use a standard depth-first search algorithm, i.e., we follow a directed path as far as possible and backtrack to the last vertex having unexplored children whenever the path cannot be extended any further. A cycle is detected as soon as an edge leads back to a vertex of the path currently under exploration. We now report on the results we can extract from this partial knowledge of $\mathcal{S}(\mathsf{M})$.

Let $S$ be the semigroup whose Hilbert basis is given by the  columns of the matrix
$$
\left[\begin{array}{rrrrrr}
1 & 0 & 0 & -2 & 1 & 2 \\
0 & 1 & 0 & -1 & -1 & -2 \\
0 & 0 & 1 & 2 & 1 & 1
\end{array}\right]\,.
$$
We have that $S$ is a child of a child of itself. This proves that the digraph $\mathcal{S}(\mathsf{M})$ with $\mathsf{M}=\mathbf{Z}^3$ contains a cycle of length 2. This disproves the Nash blowup conjecture (\cref{conj:tree}) in dimension 3 and higher. Indeed, a counterexample to the Nash blowup conjecture in dimension $n>3$ can be obtained from $S$ simply by taking the cartesian product of $S$ with the unimodular semigroup of dimension $n-3$. In \cite{CDLL3d}, the authors of this paper gave a proof of this fact that requires no computer assistance. 

This counterexample was obtained at the third iteration of the Nash blowup algorithm of the semigroup  $S$  whose Hilbert basis consists of the columns of the matrix
$$
\left[\begin{array}{rrrr}
1 & 0 & 0 & 1  \\
0 & 1 & 0 & 1  \\
0 & 0 & 1 & -6 
\end{array}\right].
$$

We have also been able to find other 3-dimensional counterexamples to the Nash blowup conjecture. For instance, the semigroup $S$ whose Hilbert basis is given by the columns of the matrix
$$
\left[\begin{array}{rrrrrr}
1 & 0 & 0 & 1 & 1 & 1 \\
0 & 1 & 1 & 2 & 2 & 2 \\
0 & 0 & 2 & -1 & 1 & 3
\end{array}\right]
$$
is one of the vertices of a cycle of length 4 in the digraph $\mathcal{S}(\mathsf{M})$. So far, we have been able to isolate cycles of lengths in the set $\{2,3,4,5,6,7\}$  inside $\mathcal{S}(\mathsf{M})$. The only limitation in finding such cycles seems to be the computational power required to find longer cycles inside a large digraph.

\subsection{Counterexamples to normalized Nash blowup in dimension 4 and higher}

We focus our presentation on the case of characteristic zero. We populated the digraph $\mathcal{K}_0(\mathsf{M})$ with $\mathsf{M}=\mathbf{Z}^4$ with 1.5 million vertices by running \cref{algo:res_tree} and we investigated the existence of cycles in it. We now report on the results we can extract from this partial knowledge of $\mathcal{K}_0(\mathsf{M})$.

Let $\sigma^\vee$ be the cone in $\mathsf{M}_\mathbf{R}$ generated by the  columns of the matrix
\begin{align} \label{eq:count-char-0}
\mathbf{B} = \left[\begin{array}{rrrrrr}
1 & 0 & 0 & 0 & 2 & 1 \\
0 & 1 & 0 & 0 & 3 & 3 \\
0 & 0 & 1 & 0 & -2 & -1 \\
0 & 0 & 0 & 1 & -1 & -1
\end{array}\right].
\end{align}
We have that $\sigma^\vee$ is a child of itself. This proves that the digraph $\mathcal{K}_0(\mathsf{M})$ with $\mathsf{M}=\mathbf{Z}^4$ contains a cycle of length 1 (a loop). This disproves the normalized Nash blowup conjecture (\cref{conj:normalized_tree}) in dimension 4 and higher. Indeed, a counterexample to the normalized Nash blowup conjecture in dimension $n>4$ can be obtained from $\sigma^\vee$ simply by taking the cartesian product of $\sigma^\vee$ with the unimodular cone of dimension $n-4$. In \cite{CDLAL}, the authors of this paper gave a proof of this fact that requires no computer assistance.

\begin{definition}\label{def:reeves}
    Let $n,j$ be positive integers.
    The Reeves cone $\rho_n(j)$ is the cone generated in $\mathsf{M}_\mathbf{R}$ with $\mathsf{M}=\mathbf{Z}^n$ by the columns of the $n\times n$ matrix 
    \[
\begin{bmatrix}
1 & 0 & \cdots & 0 & 1 \\
0 & 1 & \cdots & 0 & 1 \\
\vdots & \vdots & \ddots & \vdots & \vdots \\
0 & 0 & \cdots & 1 & 1 \\
0 & 0 & \cdots & 0 & j
\end{bmatrix}
\]
\end{definition}

The counterexample in \cite{CDLAL} was obtained at the seventh iteration of the Nash blowup algorithm of the Reeves cone $\rho_4(5)$. The \emph{index} of a rational polyhedral cone $\sigma^\vee \subset \mathsf{M}_\mathbf{R}$ is the index of the sublattice generated by the primitive ray generators of $\sigma^\vee$ inside the ambient lattice $\mathsf{M}$. By exhaustion of cases, we can prove the following result.

\begin{proposition}\label{prop:reeves_45}
The cone $\rho_4(5)$ is the only simplicial cone with index less than 6 that is not  resolved by the iteration of normalized Nash blowup. 
\end{proposition}

\begin{proof}
The list of simplicial cones with index less than 6 is in \cite[Table 4]{Ataetal}.
All of them are resolved, except for $\rho_4(5)$ which has as a descendant the cone generated by the columns of $\mathbf{B}$ in \eqref{eq:count-char-0}.
\end{proof}

\begin{remark} \
\begin{enumerate}
    \item The cone generated by the columns of the matrix $\mathbf{B}$ in \eqref{eq:count-char-0} \textbf{is the only} cone leading to a cycle of any length that we have been able to find in $\mathcal{K}_0(\mathsf{M})$ despite having already computed 1.5 million vertices of the said digraph. This is in stark contrast with the cases of $\mathcal{S}(\mathbf{Z}^3)$ and $\mathcal{K}_0(\mathbf{Z}^5)$ that we report below where we have been able to identify multiple disjoint loops. 
    
    \item The cones $\rho_4(j)$ do get resolved for $j \in \{6,7,8\}$. The cone $\rho_4(9)$ is beyond reach of our current computational capability.
    
    \item Let $p$ be a prime number different from $2$ and $3$. The cone generated by the columns of the matrix $\mathbf{B}$ in \eqref{eq:count-char-0} is also a counterexample to the normalized Nash blowup conjecture in characteristic $p$ with the same argument. Indeed, the set $\mathcal{B}$ appearing in \cref{def:nash_modified} is the same for characteristic $0$ or $p$.

    \item On the other hand, the cone generated by the columns of the matrix $\mathbf{B}$ in \eqref{eq:count-char-0}  
    does get resolved in characteristics 2 and 3 via normalized Nash blowup. Nevertheless, the cone $\rho_4(5)$ does lead to different counterexamples to the normalized Nash blowup conjecture. We refer the reader to \cite{CDLAL} for the cones that provide cycles in the digraphs $\mathcal{K}_2(\mathsf{M})$ and $\mathcal{K}_3(\mathsf{M})$, with $\mathsf{M}=\mathbf{Z}^4$. 
\end{enumerate}
\end{remark}

The toric variety $X$ associated with the cone  generated by the columns of the matrix $\mathbf{B}$ in \eqref{eq:count-char-0} exhibits rather unremarkable singularities. We now present examples of affine toric varieties with well-studied types of singularities that are not resolved by iterating the normalized Nash blowup. In each case, the failure of resolution occurs because the cone  generated by the columns of the matrix $\mathbf{B}$ in \eqref{eq:count-char-0} eventually reappears during the iteration of the normalized Nash blowup algorithm.

\begin{theorem}\label{prop:galore}
For each of the following classes of singularities over characteristic zero fields, there exists an affine toric variety exhibiting the prescribed singularity type such that the iteration of the normalized Nash blowup fails to resolve its singularities:
\begin{enumerate}
    \item Hypersurface singularities.
    \item Cyclic quotient singularities.
    \item $\mathbf{Q}$-factorial Gorenstein singularities.
\end{enumerate}
\end{theorem}

\begin{proof}
Consider the following three matrices
    \begin{equation*}
\mathbf{A}_1 =
\begin{bmatrix}
1 & 0 & 0 & 0 & 1 \\
0 & 1 & 0 & 0 & 1 \\
0 & 0 & 1 & 0 & -15 \\
0 & 0 & 0 & 1 & -5
\end{bmatrix}, \quad
\mathbf{A}_2 =
\begin{bmatrix}
1 & 2 & 0 & 4 \\
0 & 3 & 0 & 0 \\
0 & 0 & 1 & 3 \\
0 & 0 & 0 & 12
\end{bmatrix}, \quad
\mathbf{A}_3 =
\begin{bmatrix}
1 & 0 & 0 & 9 \\
0 & 1 & 0 & 10 \\
0 & 0 & 1 & 11 \\
0 & 0 & 0 & 12
\end{bmatrix}.
\end{equation*}
\begin{enumerate}
\item The cone $\sigma^\vee$ generated by the columns of the matrix $\mathbf{A}_1$ reaches the cone generated by the columns of $\mathbf{B}$ in \eqref{eq:count-char-0} after four iterations of the normalized Nash blowup. Moreover, the cardinality of the Hilbert basis of the cone $\sigma^\vee$ is five, given by its five primitive ray generators. Hence, the corresponding 4-dimensional toric variety can be embedded in $\mathbf{k}^5$ as a hypersurface. The corresponding equation in $\mathbf{k}^5$ is 
\begin{equation*}
x_1x_2 - x_3^9x_4^5x_5=0.
\end{equation*}

\item The cone $\sigma^\vee$ generated by the columns of the matrix $\mathbf{A}_2$ reaches the cone generated by the columns of $\mathbf{B}$ in \eqref{eq:count-char-0} after two iterations of the normalized Nash blowup. The PALP normal form of the dual cone $\sigma\subset \mathsf{N}_\mathbf{R}$ is generated by the columns of
\begin{equation*}
\begin{bmatrix}
1 & 0 & 0 & 4 \\
0 & 1 & 0 & 8 \\
0 & 0 & 1 & 3 \\
0 & 0 & 0 & 12
\end{bmatrix}.
\end{equation*}
This proves that the corresponding affine toric variety has cyclic quotient singularities. Indeed, a simplicial toric variety has cyclic quotient singularities if and only if the group $\mathsf{N}/\mathsf{N}_\sigma$ is cyclic, where $\mathsf{N}_\sigma$ is the sublattice of $\mathsf{N}$ generated by the primitive ray generators of $\sigma$. In this example $\mathsf{N}/\mathsf{N}_\sigma\simeq \mathbf{Z}/12\mathbf{Z}$.

\item  The cone $\sigma^\vee$ generated by the columns of the matrix $\mathbf{A}_3$ reaches the cone generated by the columns of $\mathbf{B}$ in \eqref{eq:count-char-0} after two iterations of the normalized Nash blowup. Moreover, its dual cone is the cone $\sigma\subset \mathsf{N}_\mathbf{R}$ generated by the columns of the matrix 
\begin{equation*}
\begin{bmatrix}
0 & 4 & 0 & 0 \\
0 & 0 & 6 & 0 \\
0 & 0 & 0 & 12 \\
1 & -3 & -5 & -11
\end{bmatrix}.
\end{equation*}
The point $m = [1,1,1,1]^T$ satisfies $\langle m, \rho \rangle = 1$ for every primitive ray generator $\rho$ of $\sigma$ (that is, for each column of the matrix above). Hence the corresponding toric variety has Gorenstein singularities by \cite[Theorem~3.8]{Dai02}. Moreover, the fact that $\sigma^\vee$ is simplicial is equivalent to the toric variety being $\mathbf{Q}$-factorial \cite[Section~11.4]{cox2011toric}.\qedhere
\end{enumerate}
\end{proof}

\begin{remark}
In \cite{Sampaio} it was observed that no counterexamples were known regarding the resolution properties of Nash blowups in the case of hypersurfaces. The previous theorem provides such an example.
\end{remark}

\begin{remark}
In our database of 1.5 million cones, we have not found a single instance of an isolated singularity that fails to be resolved by iterating the normalized Nash blowup.
\end{remark}

\subsection{Counterexamples to normalized Nash blowup in dimension 5}

We populated the digraph $\mathcal{K}_0(\mathsf{M})$ with $\mathsf{M}=\mathbf{Z}^5$ with 0.5 million vertices by running \cref{algo:res_tree} and we investigated the existence of cycles in it.

In stark contrast with the case of dimension 4 but in line with the case of non-normalized Nash blowup in dimension 3, in the case of dimension 5 we were able to find several different cycles providing counterexamples to the normalized Nash blowup conjecture. Indeed, we have found cycles of lengths in the set $\{1,2,4,5,6,7,8,9,10,11,12,13,14,15\}$. We do not know if the missing $3$ is only an artifact of our limited computer capabilities or their absence reveals a real property of normalized Nash blowups. Moreover, it is an intriguing question whether loops of arbitrarily large length can be found in $\mathcal{K}_0(\mathsf{M})$.

To illustrate these findings, let us exhibit the cones giving cycles of length 1 and 9. The cone generated in $\mathsf{M}_\mathbf{R}$ with $\mathsf{M}=\mathbf{Z}^5$ by the columns of the matrix
\begin{equation*}
\begin{bmatrix}
1 & 0 & 0 & 0 & 0 & 2 & 1 & 1 \\
0 & 1 & 0 & 0 & 0 & 2 & 2 & 2 \\
0 & 0 & 1 & 0 & 0 & -1 & -1 & 0 \\
0 & 0 & 0 & 1 & 0 & 1 & 1 & 0 \\
0 & 0 & 0 & 0 & 1 & -2 & -1 & -1
\end{bmatrix}.
\end{equation*}
is a child of itself; that is, it is the unique vertex of a cycle of length 1 in the digraph $\mathcal{K}_0(\mathsf{M})$. On the other hand, the cone generated by the columns of the matrix 
\begin{equation*}
\begin{bmatrix}
1 & 0 & 0 & 0 & 0 & 1 & -1 \\
0 & 1 & 0 & 0 & 0 & 3 & 0 \\
0 & 0 & 1 & 0 & 0 & 1 & 1 \\
0 & 0 & 0 & 1 & 0 & -1 & 1 \\
0 & 0 & 0 & 0 & 1 & -2 & 0
\end{bmatrix}.
\end{equation*}
is one of the vertices of a cycle of length 9 in $\mathcal{K}_0(\mathsf{M})$.

\begin{remark}
As a final comment on this section, we emphasize that the non-normalized version of the Nash blowup conjecture is significantly more demanding computationally than the normalized version. Indeed, in \cref{def:nash_modified_nonnormal}, one must compute a semigroup for every element of $\mathcal{B}$. In contrast, in \cref{def:nash_modified}, most elements of $\mathcal{B}$ are not vertices of the relevant polyhedron and therefore require no computation at the final step.
\end{remark}

\section{Further computational results and future directions}

Here we outline what we consider the main remaining challenges and comment on the progress we have made toward addressing them.

\subsection{Nash blowup in dimension two}

The remaining open questions in this case are the following: the resolution properties of the non-normalized Nash blowup in characteristic zero, even for toric varieties, and the resolution properties of the normalized Nash blowup for varieties over fields of positive characteristic, although in the toric case the conjecture is known to hold, as proved in \cite[Theorem 2.5]{DJNB}.

The case of non-normalized Nash blowup of toric surfaces over an algebraically closed field of characteristic zero has already some amount of positive evidence \cite{Reb,DuarSurf,DS}. We have further explored this case computationally: we populated the digraph $\mathcal{S}(\mathsf{M})$ with $\mathsf{M}=\mathbf{Z}^2$ with 50 thousand semigroups and every single one of them gets resolved via Nash blowup. 

\subsection{Normalized Nash blowup in dimension three}

The case of three dimensional toric varieties was previously studied in \cite{Ataetal} where they computed the resolution of 2 thousand cones and concluded that every one of them was eventually resolved via normalized Nash blowup. Continuing with their line of work, we populated the digraph $\mathcal{K}_0(\mathsf{M})$ with $\mathsf{M}=\mathbf{Z}^3$ with 6.5 million cones and reached the same conclusion: every single one of these cones gets resolved via Nash blowups.

In marked contrast with the usual toric algorithm for resolving the singularities of toric varieties, we currently lack an invariant that is guaranteed to improve under the iteration of normalized Nash blowups. As already noted in \cite{Ataetal}, every proposed invariant has eventually been refuted by a counterexample. The search for a satisfactory measure of improvement under Nash blowups continues to pose a significant challenge.

Since the Reeves cones $\rho_n(j)$ played a significant role in the counterexample we found in dimension 4, it is natural to investigate their behavior in dimension 3 as well. The following result summarizes our findings. 

\begin{proposition}
For every $1 \leq j \leq 1000$, the Reeves cone $\rho_3(j)$ is resolved after finitely many iterations of the normalized Nash blowup over a field of characteristic zero.
\end{proposition}

\begin{remark}
Proving that all Reeves cones $\rho_3(j)$ are resolved would represent a substantial advance. Except for $j = 1$, which corresponds to the unimodular cone, none of the computed Reeves cones is a descendant of any other, so there is no straightforward inductive argument available.
\end{remark}

\subsection{Further questions}

\begin{itemize}
\setlength{\itemsep}{0.5em}
    \item \emph{Unique cycle in dimension four.} It remains unknown whether there are other cycles in the digraph $\mathcal{K}_{0}(\mathsf{M})$ with $\mathsf{M=\mathbf{Z}^4}$ other than the loop we have found.

    \item \emph{One step resolution.} In every dimension greater or equal than 2, it is known that there exist infinitely many cones in $\mathcal{K}_{0}(\mathsf{M})$ that are resolved by a single normalized Nash blowup \cite{CDLLcharfree}, but a complete characterization of these cones is still lacking.

    \item \emph{Infinite families that get resolved.} Beyond the examples that are resolved by a single normalized Nash blowup, we are not aware of any systematic families, in dimension greater than two, of cones that are resolved by iterating the normalized Nash blowup. As noted above, even proving that all Reeves cones in dimension 3 are resolved by the normalized Nash blowup would already constitute a substantial advance.

    \item \emph{Isolated singularities.} None of the counterexamples found so far contradicts the possibility that isolated singularities may always be resolved by Nash blowup or normalized Nash blowup.

\end{itemize}

\appendix

\section{Normalized Nash blowup as a subdivision problem}

In the main body of the text we described affine toric varieties in terms of cones $\sigma^\vee \subset \mathsf{M}_\mathbf{R}$. However, the more common approach is to describe them via the dual cone $\sigma \subset \mathsf{N}_\mathbf{R}$. We adopted the $\mathsf{M}$-side viewpoint because it provides a unified framework for both the normalized and the non-normalized Nash blowup. The main advantage of working on the $\mathsf{N}$-side is that the globalization process, which constructs general toric varieties by gluing affine pieces, is more naturally expressed using the cones $\sigma \subset \mathsf{N}_\mathbf{R}$ rather than their duals $\sigma^\vee \subset \mathsf{M}_\mathbf{R}$. In this section, we describe the normalized Nash blowup from this perspective.

\begin{definition}[Normalized Nash blowup on the $\mathsf{N}$-side]
\label{def:nash}

Given a cone $ \sigma \subseteq \mathsf{N}_{\mathbf{R}} $ whose dual cone is $ \sigma^\vee \subseteq \mathsf{M}_{\mathbf{R}} $ the Nash subdivision of $\sigma$ is the result of the following set of steps, which form a collection of cones that is a subdivision of $ \sigma$.

\begin{description} 
\item[P1] Compute the Hilbert basis $ \mathcal{H} $ of $S = \sigma^\vee\cap \mathsf{M}$.
\item[P2] Compute the set $\displaystyle \mathcal{B} = \left\{ \sum_{h \in B} h \mid B \subseteq \mathcal{H} \text{ is a basis for }\mathsf{M}\otimes_\mathbf{Z}\mathbf{k}  \right\} $.
\item[P3] Return the normal fan of the unbounded polyhedron $\operatorname{Conv}(\mathcal{B}) + \sigma^\vee$.
\end{description}
Given a fan, the Nash subdivision is the fan obtained applying the above process to each maximal cone of the fan.
\end{definition}

The result of this process is a subdivision of the original cone. Remark that this process is equivalent to \cref{def:nash_modified} with the only difference that we gather the duals of the feasible cones in a fan in $\mathsf{N}_\mathsf{R}$. In these terms, the normalized Nash blowup conjecture becomes:

\begin{conjecture}[Normalized Nash blowup conjecture on the $\mathsf{N}$-side]
    Let $\sigma\subset \mathsf{N}_\mathbf{R}$ be a cone. Then, the iteration of the Nash subdivision process eventually stabilizes in a unimodular triangulation.
\end{conjecture}

Let us consider an example in detail. As shown in \cref{fig:12}, let $\sigma$ be the cone generated by the columns of the matrix $\mathbf{A}$ whose dual cone $\sigma^\vee$ is generated by the columns of the matrix  $\mathbf{B}$ and whose Hilbert basis $\mathcal{H}$ is composed by the columns of the matrix $\mathbf{C}$, where
\begin{equation*}
\mathbf{A} = \begin{bmatrix}
-1 & 3  \\
2 & -1  \\
\end{bmatrix}, \quad
\mathbf{B} = \begin{bmatrix}
2 & 1  \\
1 & 3  \\
\end{bmatrix}, \quad\mbox{and}\quad
\mathbf{C} = \begin{bmatrix}
2 & 1 & 1 & 1 \\
1 & 3 & 2 & 1 \\
\end{bmatrix}.
\end{equation*}

\begin{figure}[htbp]
\centering
\begin{minipage}{0.3\textwidth}
    \centering
		\begin{tikzpicture}[scale =0.5]
				\fill[gray!15] (0,0) -- (-2,4) -- (3,4) -- (3,-1) -- cycle ;

				\foreach \x in {-3,-2,...,3}
						\foreach \y in {-1,0,...,4}
								\fill (\x,\y) circle (2pt);
				\draw[color=blue,fill=blue] (0,0) circle (3pt);

				\draw[color=blue,-latex,thick] (0,0) -- (3,-1);
				\draw[color=blue,-latex,thick] (0,0) -- (-1,2);
    \end{tikzpicture}
    \caption*{Cone $\sigma$.}
\end{minipage}
\hfill
\begin{minipage}{0.3\textwidth}
    \centering
		\begin{tikzpicture}[scale = 0.5]

				\fill[gray!50] (0,0) -- (4/3,4) -- (3,4) -- (3,3/2) -- cycle ;

				\foreach \x in {-3,-2,...,3}
						\foreach \y in {-1,0,...,4}
								\fill (\x,\y) circle (2pt);
				\draw[color=red,fill=red] (0,0) circle (3pt);

				\draw[color=red,-latex,thick] (0,0) -- (1,3);
				\draw[color=red,-latex,thick] (0,0) -- (2,1);
    \end{tikzpicture}
    \caption*{Cone $\sigma^\vee$.}
\end{minipage}
\hfill
\begin{minipage}{0.3\textwidth}
    \centering
    \begin{tikzpicture}[scale = 0.5]

				\fill[gray!50] (0,0) -- (4/3,4) -- (3,4) -- (3,3/2) -- cycle ;

				\foreach \x in {-3,-2,...,3}
						\foreach \y in {-1,0,...,4}
								\fill (\x,\y) circle (2pt);
				\draw[color=red,fill=red] (0,0) circle (3pt);

				\draw[color=red,-latex,thick] (0,0) -- (1,3);
				\draw[color=red,-latex,thick] (0,0) -- (2,1);
				\draw[color=red,dashed] (1,3) -- (3,4) -- (2,1);

				\draw (1,1) circle (5pt);
				\draw (1,2) circle (5pt);
				\draw (1,3) circle (5pt);
				\draw (2,1) circle (5pt);
    \end{tikzpicture}
    \caption*{Hilbert Basis.}
\end{minipage}
\caption{The first two steps of the procedure.}
\label{fig:12}
\end{figure}

Now, in \cref{fig:three}, we show on the left the set $\mathcal{B}$ together with the cone $\operatorname{Conv}(\mathcal{B})+\sigma^\vee$. Finally, on the right, we show the Nash subdivision of the cone $\sigma$.

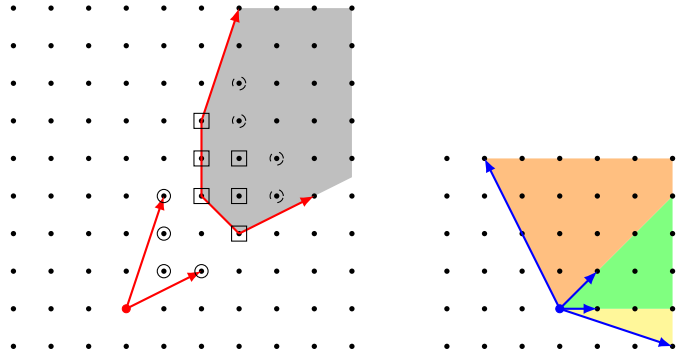
\begin{figure}[ht]
	\begin{tikzpicture}[scale = 0.5]
										
			\fill[gray!50] (3,2) -- (2,3) -- (2,5) -- (3,8) -- (6,8) -- (6,7/2) -- cycle ;

			\foreach \x in {-3,-2,...,6}
					\foreach \y in {-1,0,...,8}
							\fill (\x,\y) circle (2pt);
			\draw[color=red,fill=red] (0,0) circle (3pt);

			\draw[color=red,-latex,thick] (0,0) -- (1,3);
			\draw[color=red,-latex,thick] (0,0) -- (2,1);

			\draw[color=red,-latex,thick] (2,5) -- (3,8);
			\draw[color=red,-latex,thick] (3,2) -- (5,3);
			\draw[color=red,thick] (3,2) -- (2,3) -- (2,5);

			\draw (1,1) circle (5pt);
			\draw (1,2) circle (5pt);
			\draw (1,3) circle (5pt);
			\draw (2,1) circle (5pt);

			\draw (3-.2,2-.2) rectangle (3+.2,2+.2);
			\draw (3-.2,3-.2) rectangle (3+.2,3+.2);
			\draw (3-.2,4-.2) rectangle (3+.2,4+.2);
			\draw (2-.2,3-.2) rectangle (2+.2,3+.2);
			\draw (2-.2,4-.2) rectangle (2+.2,4+.2);
			\draw (2-.2,5-.2) rectangle (2+.2,5+.2);
			\draw[dashed] (4,3) circle (5pt);
			\draw[dashed] (4,4) circle (5pt);
			\draw[dashed] (3,5) circle (5pt);
			\draw[dashed] (3,6) circle (5pt);

	\end{tikzpicture}
	\hspace{1cm}
	\begin{tikzpicture}[scale = 0.5]
				\fill[yellow!50] (0,0) -- (3,-1) -- (3,0) -- cycle ;
				\fill[green!50] (0,0) -- (3,0) -- (3,3) -- cycle ;
				\fill[orange!50] (0,0) -- (-2,4) -- (3,4) -- (3,3) -- cycle ;

				\foreach \x in {-3,-2,...,3}
						\foreach \y in {-1,0,...,4}
								\fill (\x,\y) circle (2pt);
				\draw[color=blue,fill=blue] (0,0) circle (3pt);

				\draw[color=blue,-latex,thick] (0,0) -- (3,-1);
				\draw[color=blue,-latex,thick] (0,0) -- (1,0);
				\draw[color=blue,-latex,thick] (0,0) -- (1,1);
				\draw[color=blue,-latex,thick] (0,0) -- (-2,4);
	\end{tikzpicture}
\caption{The resulting subdivision of the original cone $\sigma$.}
\label{fig:three}
\end{figure}

\bibliographystyle{alpha}
\bibliography{ref}

\end{document}